\documentclass[review]{elsarticle}

\journal{arKiv}

\usepackage{latexsym,amsfonts,amssymb,amsmath,amsthm,euscript,subfigure}
\usepackage{graphicx}
\usepackage{epstopdf}
\usepackage{hyperref}
\usepackage{natbib}
\usepackage{color,fancybox}
\usepackage[normalem]{ulem}  
\usepackage{latexsym}
\usepackage{comment}
\usepackage[british]{babel}

\usepackage{hyperref}

\usepackage{mathrsfs}

\newtheorem{theorem}{Theorem}[section]

\newtheorem{definition}{Definition}[section]
\newtheorem{remark}{Remark}[section]

\newcommand{\new}[1]{\textcolor{black}{#1}}

\newcommand{\B}[1]{\boldsymbol{#1}}
\newcommand{\dip}{\displaystyle}

\begin{document}
	
	\begin{frontmatter}
		
		\title{\new{Existence, uniqueness, and approximation  for solutions of a functional-integral equation in $L^p$ spaces}}
		
		\author[Unesp]{Suzete M. Afonso\corref{cor}}
		\ead{smafonso@rc.unesp.br}
		\author[UFRB]{Juarez S. Azevedo}
		\ead{juarez@ufrb.edu.br}
		\author[UFRB]{Mariana P. G. da Silva}
		\ead{mpinheiro@ufrb.edu.br}
		\author[UFRB]{ Adson M. Rocha}
		\ead{adson@ufrb.edu.br}

		\address[Unesp]{Univ Estadual Paulista, Instituto de Geoci\^{e}ncias e Ci\^{e}ncias Exatas, 13506-900, Rio Claro SP, Brasil}
		\address[UFRB]{Universidade Federal do Rec\^oncavo da Bahia - UFRB, CETEC, Centro, 44380-000, Cruz das Almas-BA, Brazil.}

		\cortext[cor]{Corresponding author.}
		
		\begin{abstract} 
			In this work we consider  the general functional-integral equation:
			\begin{equation*}
			y(t) =  f\left(t, \int_{a}^{b} k(t,s)g(s,y(s))ds\right), \qquad t\in [a,b],\end{equation*}
			and give conditions that guarantee existence and uniqueness of solution  in $L^p([a,b])$, with \new{$1<p<\infty$}. 
			\new{
				We use 
				Banach Fixed Point Theorem} and
				  employ \new{the} successive approximation method and Chebyshev quadrature for approximating the values of integrals.  Finally, to illustrate the results of this work, we  provide some numerical examples. 
		\end{abstract} 
		

	\end{frontmatter}

	\section{Introduction}

	Nonlinear integral equations have been extensively studied in the literature, see for example \new{integral  equations} of Urysohn type \cite{Ibr:09, Kar:05}, Hammerstein type \cite{FigG:73}, and Volterra type \cite{kwapisz1991bielecki}; the works cited had  as a focal point conditions of existence of solution for such equations. In this sense, the theme has induced some authors to improve and extend these results to existence of solutions involving functional integral equations in the space $L^1([0,1])$ \cite{BanK:89, Ema:91, Ema:92}. For this reason, these authors have considered the equation:
	\begin{equation} \label{emanu1}
	y(t) = f\left(t, r \int_{0}^{1} k(t,s)g(s,y(s))ds\right), \qquad t\in [0,1], \, r>0,
	\end{equation} and proved the existence of  solutions of that equation in $L^1([0,1])$. 
	%
	In this way, they have concluded that equation \eqref{emanu1} has a solution in this space. 
		An extension of these results was given in $L^p([0,1])$, $p\geq 1$, by Karoui and Adel in \cite{Kar:10}, considering  nonlinear integral equations of the Hammerstein and Volterra type. Moreover, in \cite{nadir}, the authors were able to  guarantee the existence and  uniqueness of \new{the} solution of  Hammerstein integral equation in the $L^p([0,1])$ space. However, the results \new{were} limited to this  specific type of equation. In order to fill this gap,  we  consider the functional-integral equation defined by: 
	\begin{equation} \label{emanu}
\new{	y(t) = f\left(t,  \int_{a}^{b} k(t,s)g(s,y(s))ds\right), \qquad t\in [a,b], \, r>0,}
	\end{equation} 
	\new{with $a,b\in \mathbb{R}$,}
	and  prove that, under certain  hypotheses, it admits a unique solution in $L^p ([a, b])$, \new{$1<p<\infty$}.  Here, we  \new{delete} the term $r$ from our calculations \new{and consider an arbitrary real interval $[a,b]$ }. 
	
	

	As starting point,  we  show that, under certain conditions, \new{the operator defined by the right hand side of \eqref{emanu}}  maps $L^{p}([a,b])$ into itself. It ensures that any solution of \eqref{emanu} lies in $L^{p}([a,b])$. And, under additional hypotheses, we prove that eq. \eqref{emanu} has a unique solution in $L^{p}([a,b])$, which can be obtained as the limit of successive approximations.

	The remainder of the paper is organized as follows. In Section \ref{sectionresult} we present results on existence and uniqueness of  solutions for functional-integral equation, considering the successive approximation method. \new{In
		Section \ref{section_error} we exhibit an estimative of the error generated by the successive approximation method}. Numerical examples are provided in Section \ref{numericalexamples} and we conclude the paper in  Section \ref{conclusion}.
	
	\section{Main Results}\label{sectionresult}
	
	In what follows, we assume that the function $f: [a,b]\times \mathbb{R} \to \mathbb{R}$ in \eqref{emanu} satisfies the Caratheodory conditions, that is,
	\begin{enumerate}
		\item[$i)$]   $f(t,x)$ is continuous in $x$  for each fixed $t$;
		\item[$ii)$] $f(t,x)$ is measurable in $t$  for each fixed $x$;
		\item[$iii)$] there is a non-negative Lebesgue-integrable function  $m :[a,b]\to \mathbb{R}$ such that $|f(t, x)| \leq m (t)$, for all $(t,x)\in [a,b]\times \mathbb{R}$.
	\end{enumerate}

	\begin{theorem} \label{th1}
		Assume that the following conditions are satisfied:
		\begin{enumerate}
			\item[$(A1)$]\new{There are a non-negative function $h_1\in L^{p}([a,b])$ and a non-negative constant $b_1$ such that
			\[
			|f(t,x)| \leq h_1(t) + b_1|x|^{q/p} \qquad \textrm{for a.e.}\,\, t\,\,\textrm{in}\,\, [a,b], \, x\in \mathbb{R},~ \mbox{and}~ \frac{1}{p}+\frac{1}{q}=1.
			\]}
			\item[$(A2)$] \new{The kernel \new{$k(t,\cdot )$} is measurable, belongs to the space $L^{q}([a,b])$ {for all $t\in [a,b]$} and}
			\begin{equation}\label{h1}
		\new{	\left(\int_{a}^{b} |k(t,s)|^q ds \right)^{\frac{1}{q}} \leq M_1(t), \qquad \textrm{for any}\,\,\, t\in [a,b],}
			\end{equation}
			\new{where $M_1$ is a non-negative function in $L^{p}([a,b])$.}
			\item[$(A3)$] The function \new{$g(s,z)$} is a map from $[a,b]\times \mathbb{R}$ into $[a,b]$
			satisfying Caratheodory conditions and such that \[\textcolor{black}{|g(s,z)|\leq a_{0}(s)+b_{0}|z|}, \]
		\new{where $a_0$ is a non-negative function in $L^{p}([a,b])$ and $b_0$ is a non-negative constant.}
		\end{enumerate}
		Under conditions $(A_1)$, $(A_2)$, and $(A_3)$, the operator 
		\begin{equation}\label{operatorA}
		(Ay)(t) =  f\left(t, \int_{a}^{b} k(t,s)g(s,y(s))ds\right), \qquad t\in [a,b],
		\end{equation}
		is a map from $ L^{p}([a,b])$  into  $L^{p}([a,b])$.
	\end{theorem}

	\begin{proof}
		
		\new{Firstly}, note that 	\textcolor{black}{$\|g(\cdot,y(\cdot))\|_p$$< \infty$ whenever $y\in  L^{p}([a,b])$}. Indeed, by Condition $(A3)$, we have
		\[
	\new{	|g(s,y(s))|^p \leq \left(a_{0}(t)+b_{0}|y(s)|\right)^p, \qquad \textrm{for all}\,\, s\in [a,b], }
		\]
		and, therefore,
		\[
		\textcolor{black}{\|g(\cdot,y(\cdot))\|_p}= \left(\int_{a}^{b} |g(s,y(s))|^p ds  \right)^{\frac{1}{p}} \leq \left(\int_{a}^{b}\left(a_{0}(s)+b_{0}|y(s)|\right)^p ds \right)^{\frac{1}{p}}.
		\]
		Using  Minkowski's inequality, we get
		\[
		\textcolor{black}{\|g(\cdot,y(\cdot))\|_p} \leq  \left(\int_{a}^{b} |a_0(s)|^p ds\right)^{\frac{1}{p}} + \left(\int_{a}^{b} b_0^p|y(s)|^p ds\right)^{\frac{1}{p}},
		\]
		whence it follows that
		\begin{equation}\label{glq}
		\textcolor{black}{\|g(\cdot,y(\cdot))\|_p} \leq  \|a_0\|_p + b_0 \|y\|_p <\infty.
		\end{equation}

		Now let us show that if $y\in L^{p}([a,b])$ then $Ay\in L^{p}([a,b])$. In fact, for $y\in L^{p}([a,b])$ and
		$t\in [a,b]$, \new{it follows from Condition $(A_1)$ that}
		\begin{eqnarray*}
			|Ay(t)|^p & = & \left|f\left(t, \int_{a}^{b} k(t,s)g(s,y(s))ds\right)\right|^p \\
			& \leq & \left[ h_1(t) + b_1\left|\int_{a}^{b} k(t,s)g(s,y(s))ds \right|^{q/p}\right]^p \\
			& \leq & \left( 2 \max \left\{h_1(t), b_1\left|\int_{a}^{b} k(t,s)g(s,y(s))ds \right|^{q/p}\right\}\right)^p\\
			& = & 2^p  \max \left\{[h_1(t)]^p, b_1^p\left| \int_{a}^{b} k(t,s)g(s,y(s))ds \right|^q\right\}\\
			&\leq & 2^p \left([h_1(t)]^p + b_{1}^{p}\left| \int_{a}^{b} | k(t,s)g(s,y(s))| ds \right|^q\right).\\
		\end{eqnarray*}

		By H\"older's inequality, \new{condition $(A_2)$, and eq. \eqref{glq}}, we have

		\vspace{0,3 cm}
		
		$\displaystyle\begin{array}{lll}
		\displaystyle\left|\int_{a}^{b}  k(t,s)g(s,y(s))ds\right|&\leq& \displaystyle\int_{a}^{b} | k(t,s)g(s,y(s))| ds\\
		&\leq&\left(\int_{a}^{b}|k(t,s)|^{q}ds\right)^{1/q}\left(\int_{a}^{b}|g(s,y(s))|^{p}ds\right)^{\frac{1}{p}}\\
	&\leq &	\new{M_1(t) \left(\int_{a}^{b}|g(s,y(s))|^{p}ds\right)^{\frac{1}{p}}.}
		\end{array}$
		
		\vspace{0,3 cm}
		
		\new{For the sake of simplicity}, denote $N_1(t) = M_1(t) \left(\int_{a}^{b}|g(s,y(s))|^{p}ds\right)^{\frac{1}{p}}$.  Thus,
		\[
		\displaystyle\left|\int_{a}^{b}k(t,s)g(s,y(s))ds\right|^{p} \leq  \new{[N_1(t)]^p}\]
		and
		\begin{eqnarray*}
			|Ay(t)|^{p}&\leq & 2^{p} [h_{1}(t)]^{p}+2^{p}b_{1}^{p}\displaystyle\left|\int_{a}^{b}k(t,s)g(s,y(s))ds\right|^{p}\\
			&\leq&\displaystyle 2^{p}[h_{1}(t)]^{p}+2^{p}b_{1}^{p} \new{[N_1(t)]^p}.
		\end{eqnarray*}

		Finally, \new{as $N_1\in L^p([a,b])$, we obtain}
		\begin{equation*}
		\int_{a}^{b}|A(y(t))|^{p}dt \leq 2^p \|h_{1}\|_p^p + 2^{p}b_{1}^{p} \new{\|N_{1}\|_p^p} <\infty,
		\end{equation*}
		which completes the proof.
	\end{proof}
	
	\vspace{0,3 cm}
	\new{Theorem \ref{th1} states that, under certain conditions,  $Ay \in L^p[a,b]$ if $y \in L^p[a,b]$. In this way, we look for solutions  of integral equation \eqref{emanu}  in $L^{p}([a,b])$.	Now, we would like to know which conditions  are required for $f$, $k$ and $g$, in order to guarantee existence of solution this integral equation.}
	
	
	\begin{theorem} \label{exresult}Suppose that conditions {$(A1)$, $(A2)$, and $(A3)$} are satisfied. Furthermore, assume that:
		\begin{enumerate}
			\item[$(H1)$] the function $f: [a,b]\times \mathbb{R} \to \mathbb{R}$ satisfies Lipschitz  condition in the second variable, that is, there is $M>0$ such that
			\[
			|f(t,x_1) - f(t,x_2)| \leq M |x_1 - x_2|, \qquad \textrm{for any}\,\,\, t\in [a,b]\,\,\, \textrm{and}\,\,\,  x_1,x_2\in \mathbb{R}.
			\]
			
			\item [$(H2)$] \new{ the function $g: [a,b]\times \mathbb{R} \to \mathbb{R}$ satisfies Lipschitz  condition in the second variable, that is, there is $L>0$ such that}
			\[
				\textcolor{black}{|g(s,z_1) - g(s,z_2)| \leq L |z_{1} - z_{2}|, \qquad \textrm{for any}\,\,\, s\in [a,b]\,\,\, \textrm{and}\,\,\,  z_1, z_2\in \mathbb{R}.}
			\]
		\end{enumerate}
		Under such hypotheses,  the successive approximation
		\begin{equation}\label{recurrence_relation}
		\left\{ 
		\begin{aligned}
		&\new{y_0(t) = 0}, \\ 
		&y_{n+1}(t) =  f\left(t, \int_{a}^{b} k(t,s)g(s,y_n(s))ds\right), \quad n = 0, 1, 2, 3, \dots,  
		\end{aligned}
		\right. 
		\end{equation}		
		converges almost everywhere to the \new{exact solution} of eq. \eqref{emanu} provided
		\begin{equation}\label{Np_condicao}
		 \int_{a}^{b} \new{C}^p[M_1(s)]^p ds := N^p  < 1, \qquad \new{\text{where}\,\,\,\, C= ML.}
		\end{equation}	
	\end{theorem}
	
	\begin{proof}
		For this method, we put $y_0(t)$ as \new{the} identically null function and successively
		\begin{equation}\label{succapprox} 
		y_{n+1}(t) = f\left(t, \int_{a}^{b} k(t,s)g(s,y_n(s))ds\right), \quad { n = 0, 1, 2, 3, \dots}.
		\end{equation}

		\new{Since $y_0(t)\equiv 0$, it is easy to verify that $\|y_1\|_p <\infty$
		(see  Theorem \ref{th1}).}
		
		
		Using H\"older's inequality, conditions \new{$(A1)$, $(A2)$, $(A3)$, $(H1)$, and $(H2)$}, \new{we obtain, for $n\geq 1$ and $t\in [a,b]$,} 
		\begin{equation*}
		\begin{split}
		|y_{n+1}(t) - y_{n}(t)| &\leq M \int_{a}^{b} | k(t,s)| |g(s,y_n(s)) - g(s,y_{n-1}(s))| ds  \\
		& \leq  \new{C} \int_{a}^{b} | k(t,s)||y_n(s) - y_{n-1}(s))| ds \\
		& \leq  \new{C} \left(\int_{a}^{b} |k(t,s)|^q ds \right)^{\frac{1}{q}} \cdot\left(\int_{a}^{b} |y_n(s) - y_{n-1}(s))|^p ds \right)^{\frac{1}{p}}.
		\end{split}
		\end{equation*}
	\new{	Thus, for  $t\in [a,b]$, }
		\begin{equation}\label{apsuc}
		|y_{n+1}(t) - y_{n}(t)|^p \leq   \new{C}^p[M_1(t)]^p  \left(\int_{a}^{b} |y_n(s) - y_{n-1}(s))|^p ds \right).
		\end{equation}

		{Let $\mathcal{K} = \|y_1\|_p$.} \new{Inequality \eqref{apsuc}  implies}
		\begin{equation*}
		|y_{2}(t) - y_{1}(t)|^p \leq  \new{C}^p[M_1(t)]^p  \int_{a}^{b} |y_1(s)|^p ds = \new{C}^p[M_1(t)]^p   \mathcal{K}^p,
		\end{equation*}
		\begin{equation*}
		\begin{split}
		|y_{3}(t) - y_{2}(t)|^p & \leq  \new{C}^p[M_1(t)]^p   \int_{a}^{b} \new{C}^p[M_1(s)]^p   \mathcal{K}^p ds\\
		& = \new{C}^p[M_1(t)]^p \mathcal{K}^pN^{p},
		\end{split}
		\end{equation*}
		and successively
		\begin{equation*}
		|y_{n+1}(t) - y_{n}(t)|^p \leq   \new{C}^p[M_1(t)]^p \mathcal{K}^p N^{(n-1)p},
		\end{equation*}
		which is equivalent to
		\begin{equation} \label{expapsuc}
		|y_{n+1}(t) - y_{n}(t)| \leq   \new{C} M_1(t)  \mathcal{K} N^{n-1}, \quad \textrm{for}\,\,\, n\geq 1, \,\, t\in [a,b].
		\end{equation}
		Expression \eqref{expapsuc} shows  
        \new{that the  sequence $(y_n(t))$ is a  Cauchy sequence.  Using this contractivity, we can verity that the series:
				  \begin{eqnarray*}	
				  \dip\sum_{n=0}^{\infty}(y_{n+1}(t)-y_n(t)), \quad\ t\in [a,b],
				  \end{eqnarray*}
				 }		
		 has the majorant
		\[
		\new{C}  M_1(t)  \mathcal{K} (1 + N + N^2 + \cdots + N^{j-1} + N^j + \cdots), \quad\ t\in [a,b].
		\]
	\new{	Since this series converges on $L^p$-norm, 
		the convergence of the 
		sequence $(y_n(t))$
		to the exact
		solution of \eqref{emanu} is guaranteed by Banach Fixed Point Theorem \cite{atkinson2001,zeidlernonlinear}}. 
	\end{proof}
	
	Following the ideas of \cite{biel1}, \cite{biel2} and \cite{kwapisz1991bielecki}, we  prove that eq. \eqref{emanu} has a unique solution in $L^{p}([a,b])$. We   assume that conditions { $(A1)$, $(A2)$, $(A3)$, $(H1)$, and $(H2)$} {from} Theorems \ref{th1} and \ref{exresult} are satisfied. 
	
	Let $\omega:[a,b]\rightarrow \mathbb{R}^+$ be a continuous function such that $\omega(x)>0$, for all $x\in[a,b]$. Put
	\begin{equation}\label{measure}
	\| u\|_{p,\omega}=\left(\sup\left\{\omega^{-1}(x)\int_{a}^{x}|u(s)|^{p}ds; ~x\in[a,b]\right\}\right)^{1/p}.
	\end{equation}
	
	Note that for $\omega \equiv 1$ we have the classical norm in $L^{p}([a,b])$. Furthermore, it is easy to verify that eq. \eqref{measure} defines a norm for any positive continuous function  $\omega:[a,b]\rightarrow \mathbb{R}^+$ (see \cite{kwapisz1991bielecki}) and
	\begin{equation}\label{equiv}
	c_1 \|u\|_p \leq \| u\|_{p,\omega}\leq c_2\|u\|_p,
	\end{equation}
	where
	\[
	c_1 = \left(\sup\{\omega(x) : x\in [a,b]\}\right)^{-\frac{1}{p}} \quad \textrm{and} \quad c_2 = \left(\inf\{\omega(x) : x\in [a,b]\}\right)^{-\frac{1}{p}}.
	\]
	
	From \eqref{equiv}, it follows that $L^{p}([a,b])$ equipped with the norm $\| . \|_{p,\omega}$ is a Banach space, since $(L^{p}([a,b]), \|.\|_p)$ is a Banach space.

	We define $\omega: [a,b]\rightarrow \mathbb{R}^+$ by
	\begin{equation}\label{omega}
	\omega(x)=\exp^{{\lambda\int_{a}^{x}[M_1(s)]^p}ds},
	\end{equation}
	where $\lambda>1$ and
	\begin{equation}\label{c1}
	\new{C}^{p}\int_{a}^{b}(M_{1}(t))^{p}dt<\frac{1}{\lambda}.
	\end{equation}
	
	\new{We recall that $C = ML$, with $M$ and $L$ from conditions $(H1)$ and $(H2)$.}
	
	The next result is crucial to guarantee the uniqueness of solution for eq. \eqref{emanu}.

	\begin{theorem}\label{th uniqueness} Assume that conditions { $(A1)$, $(A2)$, $(A3)$, $(H1)$, and $(H2)$} from Theorems \ref{th1} and \ref{exresult} are satisfied.
		Then the operator $$Ay(t)=f\left(t,\int_{a}^{b}k(t,s)g(s,y(s))ds\right)$$ is a contraction in $L^{p}([a,b])$ with respect to the norm $\| . \|_{p,\omega}$, where $\omega$  is defined in \eqref{omega}.
	\end{theorem}
	
	\begin{proof}
		
		Theorem \new{\ref{th1}} ensures that $Ay\in L^{p}([a,b])$. \new{Furthermore, Conditions $(H1)$ and $(H2)$ imply}
		\begin{eqnarray*}
		|Ay_{1}(t)-Ay_{2}(t)|^{p} & = & \left|f\left(t,\int_{a}^{b}k(t,s)g(s,y_{1}(s))ds\right)- 
		f\left(t,\int_{a}^{b}k(t,s)g(s,y_{2}(s))ds\right)\right|^{p}\\
		&\leq &  M^{p}\left|\int_{a}^{b}k(t,s)g(s,y_{1}(s))ds-\int_{a}^{b}k(t,s)g(s,y_{2}(s))ds\right|^p\\
		&\leq & \displaystyle M^{p}\left|\int_{a}^{b}|k(t,s)||g(s, y_{1}(s))-g(s, y_{2}(s))| ds\right|^p\\
    		&\leq &  \new{C}^{p} \left|\int_{a}^{b} | k(t,s)| |y_1(s) - y_{2}(s))| ds\right|^p. 
			\end{eqnarray*}
		\new{Using H\"older's inequality, we have}
		\[
	\new{	\left|\int_{a}^{b} | k(t,s)| |y_1(s) - y_{2}(s))| ds\right|^p \leq \left(\int_{a}^{b} |k(t,s)|^q ds \right)^{p/q}\cdot \left(\int_{a}^{b} |y_1(s) - y_{2}(s))|^p ds \right).}
		\]
		Thus, 
		\[
		\new{\displaystyle|Ay_{1}(t)-Ay_{2}(t)|^{p} \leq \displaystyle \new{C}^{p} \left(\int_{a}^{b} |k(t,s)|^q ds \right)^{p/q}\cdot \left(\int_{a}^{b} |y_1(s) - y_{2}(s))|^p ds \right).}
		\]

		From \eqref{h1}, it follows that
		\begin{equation}\label{h3}
		|Ay_{1}(t)-Ay_{2}(t)|^{p} \leq \displaystyle \new{C}^{p}(M_{1}(t))^{p}\int_{a}^{b}|y_{1}(s)-y_{2}(s)|^{p}ds.
		\end{equation}
		
		Integrating both sides of \eqref{h3} with respect to $t$ over $[a,x]$, $x\in [a,b]$, and using \eqref{c1}, we obtain
		\begin{equation*}
		\begin{split}
		\displaystyle\int_{a}^{x}|Ay_{1}(t)-Ay_{2}(t)|^{p}dt&\leq\displaystyle\int_{a}^{x}\left[(\new{C}^{p}(M_{1}(t))^{p}.\int_{a}^{b}|y_{1}(s)-y_{2}(s)|^{p}ds\right]dt\\
		&=\displaystyle\int_{a}^{x}\left[(\new{C}^{p}(M_{1}(t))^{p}\exp^{\lambda\int_{a}^{t}[M_1(s)]^pds}\cdot\exp^{-\lambda\int_{a}^{t}[M_1(s)]^pds}\right.\\ &\cdot\left.\int_{a}^{b}|y_{1}(s)-y_{2}(s)|^{p}ds\right]dt\\
		&\leq \displaystyle\new{C}^{p}.\| y_{1}-y_{2}\|^{p}_{p,\omega}\cdot\exp^{\lambda\int_{a}^{x}[M_1(s)]^pds}\int_{a}^{x}(M_{1}(t))^{p}dt\\
		&\leq\displaystyle\frac{1}{\lambda}\cdot\exp^{\lambda\int_{a}^{x}[M_1(s)]^pds}\cdot\| y_{1}-y_{2}\|^{p}_{p,\omega},
		\end{split}
		\end{equation*}
		which implies
		$$\exp^{-\lambda\int_{a}^{x}[M_1(s)]^pds}\int_{a}^{b}|Ay_{1}(t)-Ay_{2}(t)|^{p}dt\leq\frac{1}{\lambda}\| y_{1}-y_{2}\|_{p,\omega}^{p}.$$
		
		Therefore,
		$$\| Ay_{1}(t)-Ay_{2}(t)\|_{p,\omega}^{p}\leq\frac{1}{\lambda}\| y_{1}-y_{2}\|_{p,\omega}^{p},$$
		whence we can conclude that
		$$\| Ay_{1}(t)-Ay_{2}(t)\|_{p,\omega}\leq\alpha\| y_{1}-y_{2}\|_{p,\omega},$$ where $\alpha=\lambda^{-1/p}<1$.
	\end{proof}
	
	Our main result will be presented in the next lines.
	Its proof  is an immediate consequence of  \new{ Banach Fixed Point Theorem}, Theorem \ref{th1}, and Theorem \ref{exresult}.

	\begin{theorem}\label{th_ex_uni} Assume that conditions \new{ $(A1)$, $(A2)$, $(A3)$, $(H1)$, and $(H2)$} from Theorems \ref{th1} and \ref{exresult} are satisfied.  Then eq. \eqref{emanu} has a unique solution in $L^{p}([a,b])$, which can be obtained as the
		limit of successive approximations.
	\end{theorem}

	\section{Error Analysis}\label{section_error}
	Consider integral equation \eqref{emanu}, where $k$, $g$, $f$, and $y$ satisfy the hypotheses from Theorem \ref{th_ex_uni}. In order to obtain the successive approximation for the exact solution, we use the recurrence relation given by: 
			\begin{equation}\label{sequenciaLp}
				\left\{ 
				\begin{aligned}
					&y_0 = 0, \\ 
					&y_{n+1} = Ay_n, \quad n = 0, 1, 2, 3, \dots,  
				\end{aligned}
				\right. 
			\end{equation}
where $A$ is the operator defined by eq. \eqref{operatorA}.			
	\begin{remark}
		Although the above equation uses quadrature rules, we suppose to choose the number of integration points in such a way that the quadrature rule will not interfere with the successive approximation error, i.e, we assume the exact integration.
	\end{remark}
	From Theorem \ref{exresult}, we have that the sequence	 $(y_n)$ 
	  converges to the exact solution since \eqref{Np_condicao} holds.
     The following theorem establishes an estimative of the error generated by the successive approximation method of this sequence.
\begin{theorem}\label{Teorema_sequencia}
Assume that the conditions from Theorem \ref{th_ex_uni} are satisfied. Then the sequence $(y_n)$ generated by the successive approximation method \eqref{sequenciaLp} satisfies the following inequality:	
\begin{equation}\label{EqThm3}
\|y^\ast - y_n\|_p \leq \dfrac{N^n}{1-N}\|y_{1}\|_p,
\end{equation}	
where $y^\ast$ is the exact solution of \eqref{emanu}.
\end{theorem}	
\begin{proof}
Following the same steps of Theorem \ref{exresult} we have 
\begin{equation*}
|y_{n+1}(t) - y_{n}(t)|^p \leq   C^p[M_1(t)]^p \|y_{1}\|_p^p N^{(n-1)p},
\end{equation*} 
so that
\begin{equation*}
\|y_{n+1} - y_{n}\|_p \leq N^{n} \|y_{1}\|_p^p.
\end{equation*} 
For $k>n$, we have
\begin{equation*}
\begin{split}
\|y_{k} - y_{n}\|_p &\leq \|y_{k} - y_{k-1}\|_p + \ldots + \|y_{n+1} - y_{n}\|_p\\
&\leq N^{k} \|y_{1}\|_p + \ldots + N^{n} \|y_{1}\|_p\\
&= (N^{k} + \ldots + N^{n}) \|y_{1}\|_p\\
&= \frac{N^n-N^k}{1-N} \|y_{1}\|_p.
\end{split}
\end{equation*} 
Making $k\rightarrow \infty$, we arrive at the desired result.

\end{proof}

	\section{Numerical Examples}\label{numericalexamples}
	In this section we describe some of the numerical experiments performed
	in solving the functional integral eq. \eqref{emanu},  which can be treated by our  Theorem \ref{th_ex_uni} 
	to illustrate the results of  existence  and uniqueness. 
	\new{For the numerical application, we use  Picard  iterative process (see Appendix A) and  admit that the convergence  is achieved when the stopping criterion  has tolerance $tol = 1e-12$ on $L^2$-norm}.
	%
	%
	%
	 \new{We employ the} MATLAB package Chebpack  available at the Mathworks
	website 
	\url{https://www.mathworks.com/matlabcentral/fileexchange/32227-chebpack} as a stand-alone algorithm for solving nonlinear systems and  investigating the performance of the numerical solution.
	
\section*{Example 1} Consider the nonlinear functional integral equation:
	\begin{equation}\label{exemplo1}
	y(t) =  \sin\left(\int_{0}^{1}(t-x)(y(x))\,dx + (t -1)\cos(1) + \sin(1)\right), \quad t\in [0,1],
	\end{equation}
	with exact solution $y(t) = sin(t)$.  Take $k(t,x) = t-x$ and $f(t,z)= \sin((t -1)\cos(1) + \sin(1) + z)$.
	\new{It is easily verified  in Theorem \ref{th_ex_uni}   that the hypotheses  are valid}. In this way, we have the guarantee of existence and uniqueness of the solution.

	 \new{To establish the minimum number of integration points  in terms of absolute errors, we note that, from 10 points of integration, we get the same convergence point with more or less iterations (see Fig. \ref{IntNumberEx1}). It allowed us to conclude that 10 points of integration are sufficient to preserve the convergence of the method. In the next experiment}, we take $10$ integration points and numerical solution {putting $n = 1,\ 2,\ 3,\ 5$, and  $10$ iterations}   on the successive approximation method.  The solutions   are compared with the exact solution $y(t) = \sin(t)$ as described graphically in Fig. \ref{solEx1}.   Already, Fig. \ref{ErroEx1}  depicts the decay of the error on $L^2$-norm  of the approximate solution considering a variation in the  iterations number $n$, from  $1$ to $20$, while in Table \ref{tabEx1} we present some values associated with these iterations. \new{The results confirm the  accuracy of the successive approximation method.}
	 
	\begin{figure}[htbp]
		\centering 
		\subfigure[]{\label{IntNumberEx1}\includegraphics[scale=.28]{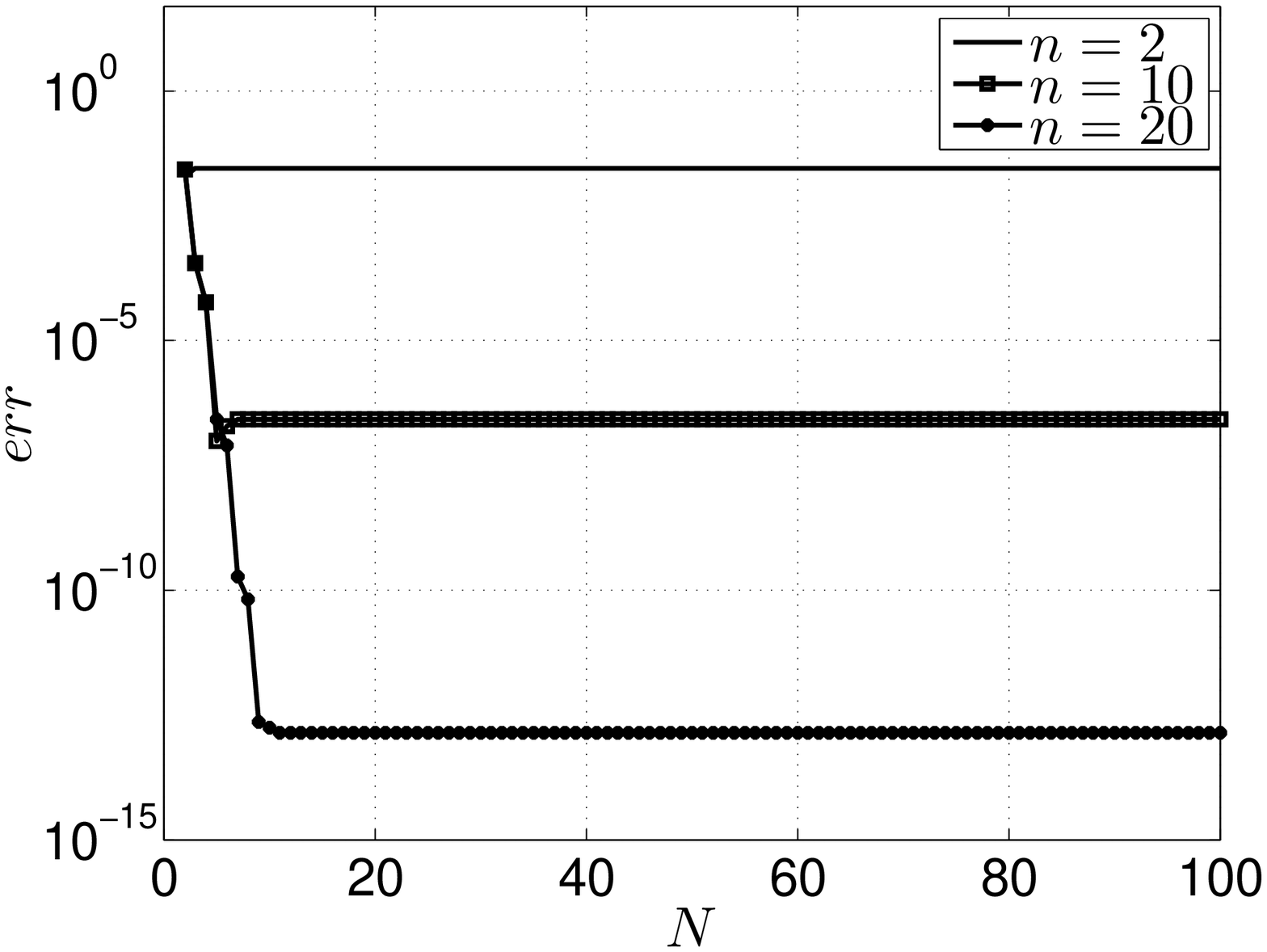}}
		\subfigure[]{\label{solEx1}\includegraphics[scale=.28]{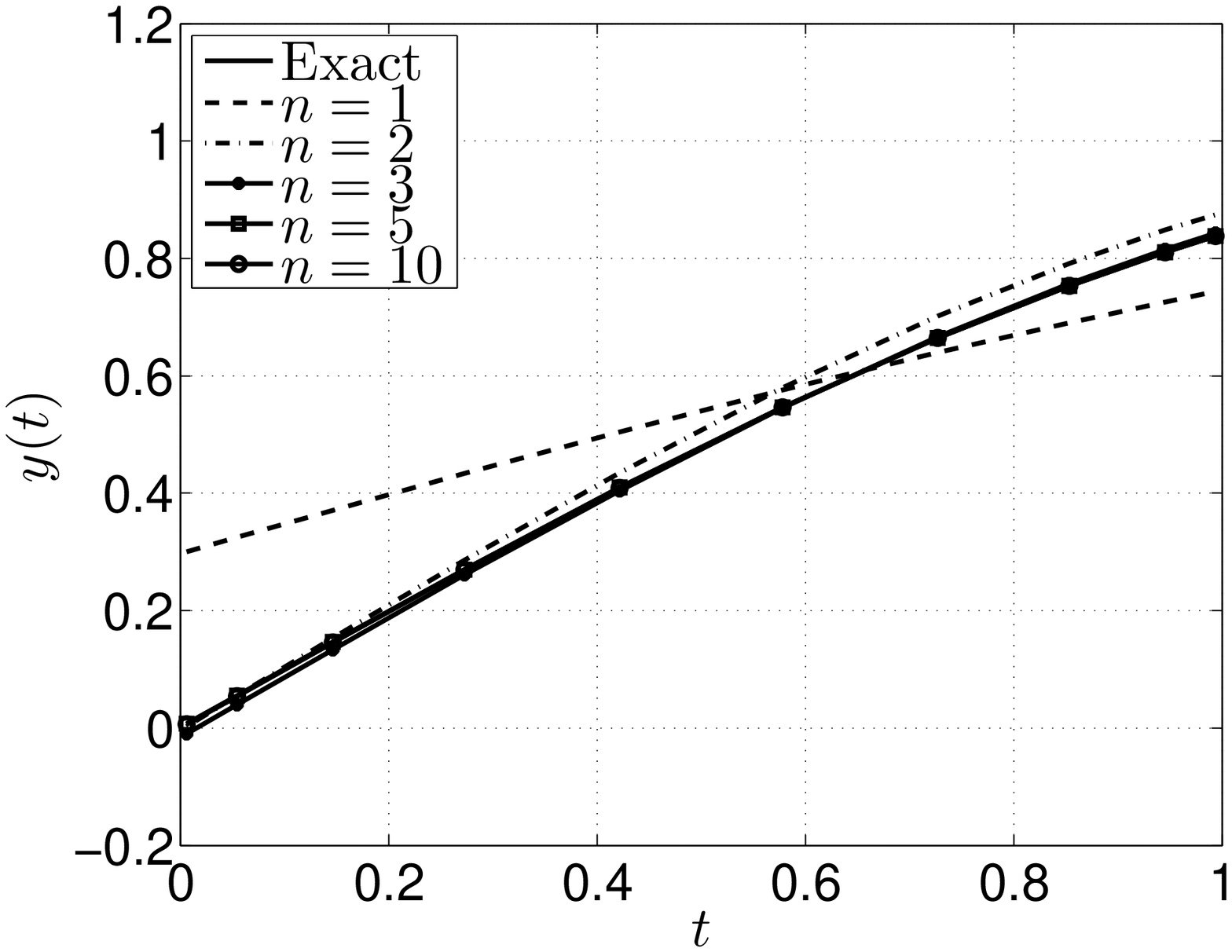}}\\ 
		\subfigure[]{\label{ErroEx1}\includegraphics[scale=.28]{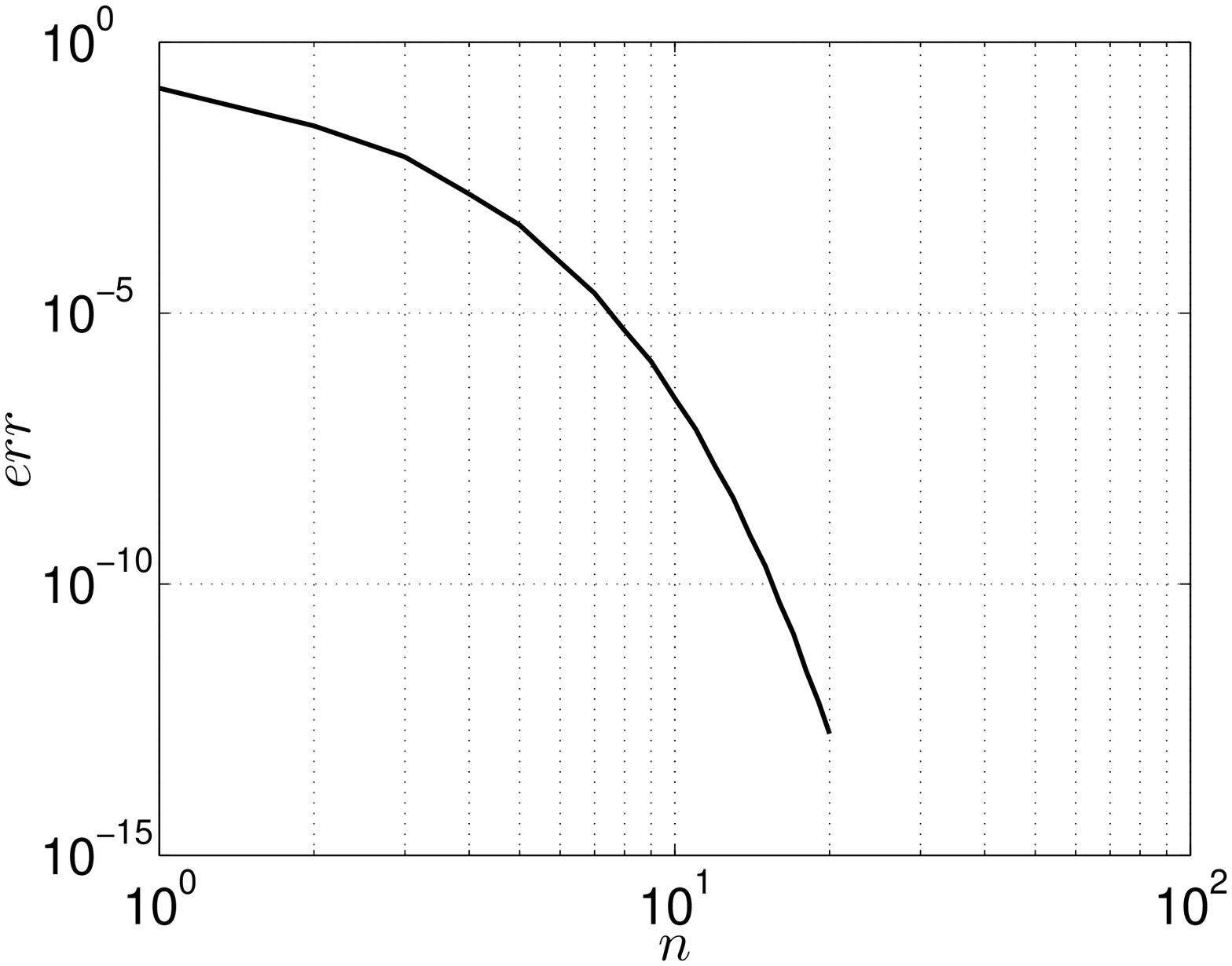}}
	\caption{(a) Absolute error (on $L^2$-norm) of the numerical solution of eq. \eqref{exemplo2} in relation to number of integration points  $N$ putting the  iterations $n = 2, 10$, $20$, and using semi-log  scale; (b) comparison of numerical solution, and exact solution, in $10$ integration points; (c) absolute error (on $L^2$-norm) of the numerical solution  of eq. \eqref{exemplo2} on iterations number $n$ from $1$ to $20$ using log-log scale.}\label{funct_Ex1}
	\end{figure}

\begin{table}[htbp] 
	\begin{center}
		\begin{tabular}{rrl}
			\hline
			Iteration (n) & Error in the $L^2$-norm (err)\\
			\hline 
			1 & 0.139055224218022&\\
			2 & 0.28090214152532&e-01\\
			3 & 0.7514001013338&e-02\\
			4 & 0.1557774788458&e-02\\
			5 & 0.417391135550&e-03\\
			6 & 0.86414955296&e-04\\
			10 & 0.265922221&e-06\\
			12 & 0.14751583&e-07\\
			15 & 0.219260&e-09\\
			20 & 0.174 &e-12\\
			\hline 
		\end{tabular}
	\end{center}
	\caption{Error in the $L^2$-norm of the  approximate solution  with respect to eq. \eqref{exemplo1} for the iterations $n=1,2,\ldots,20$.  }\label{tabEx1}
\end{table}


\newpage

\section*{Example 2} Consider the nonlinear functional integral equation:
\begin{equation}\label{exemplo2}
y(t) = \frac{1}{t+1} \log\left(\int_{0}^{1}(tx)\arctan(y(x))\,dx -\frac{t}{3} + \exp(-t-1)\right) + \tan(t) + 1,
\end{equation}
with $t\in [0,1]$ and  exact solution $y(t) = \tan(t)$.  Consider $k(t,x) = tx$ and $$f(t,z)= \frac{1}{t+1} \log\left(z-\frac{t}{3} + \exp(-t-1)\right) + \tan(t) + 1.$$
In this example, \new{the hypotheses from Theorem \ref{th_ex_uni} are also easily checked.}


Similar to the previous experiment, in Figs. \ref{funct_Ex2} we plot the approximate solutions of eq. \eqref{exemplo2} and the error associated with $L^2$-norm. The numerical solution has a good agreement with the exact solution. In Table \ref{tabEx2} we exhibit again some numerical results of this  error on $L^2$-norm.

\begin{figure}[htbp]
	\centering 
	\subfigure[]{\label{IntNumberEx2}\includegraphics[scale=.28]{fig_intN_solution_funct_Ex1.eps}}
	\subfigure[]{\label{solEx2}\includegraphics[scale=.28]{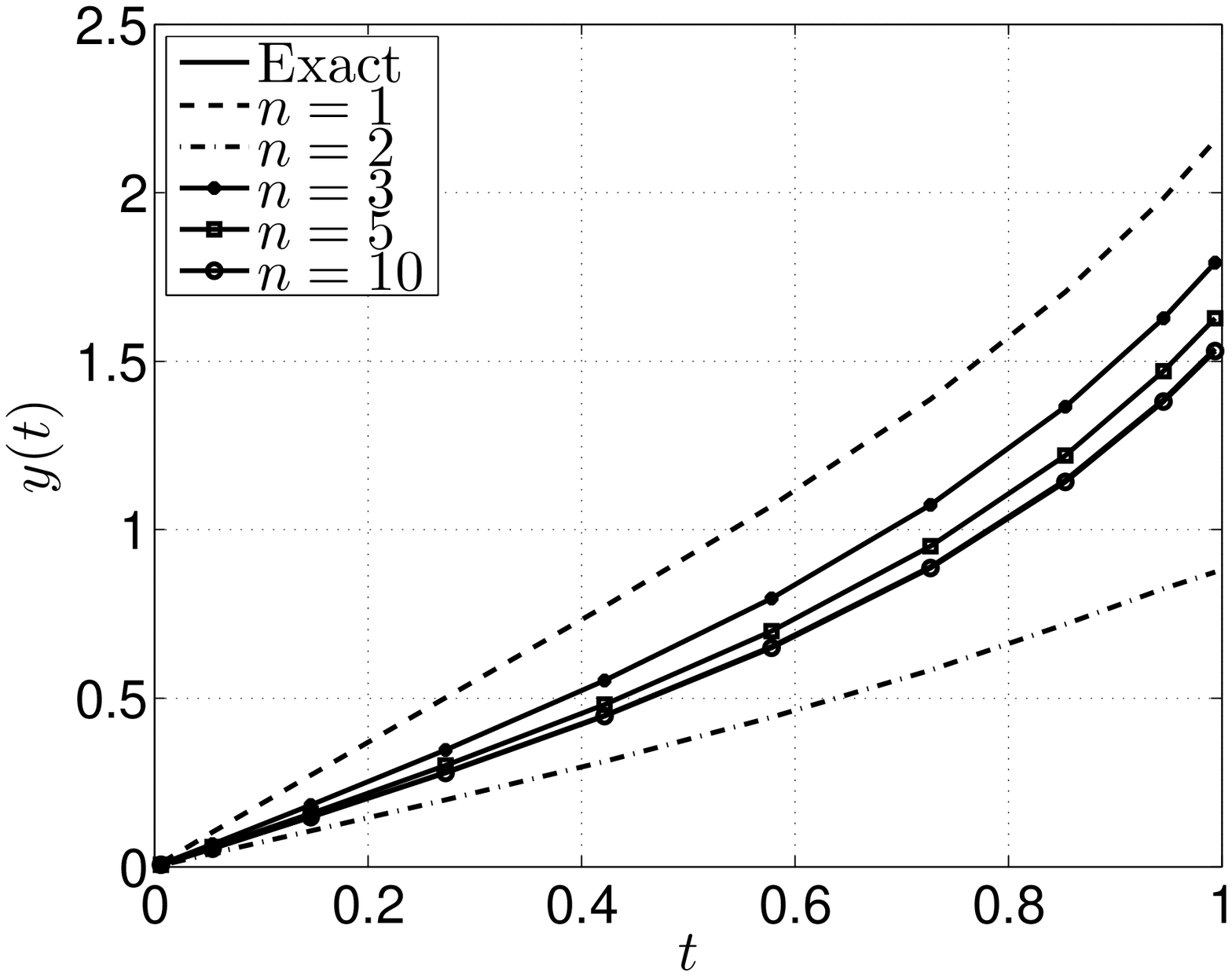}} 
	\subfigure[]{\label{ErroEx2}\includegraphics[scale=.28]{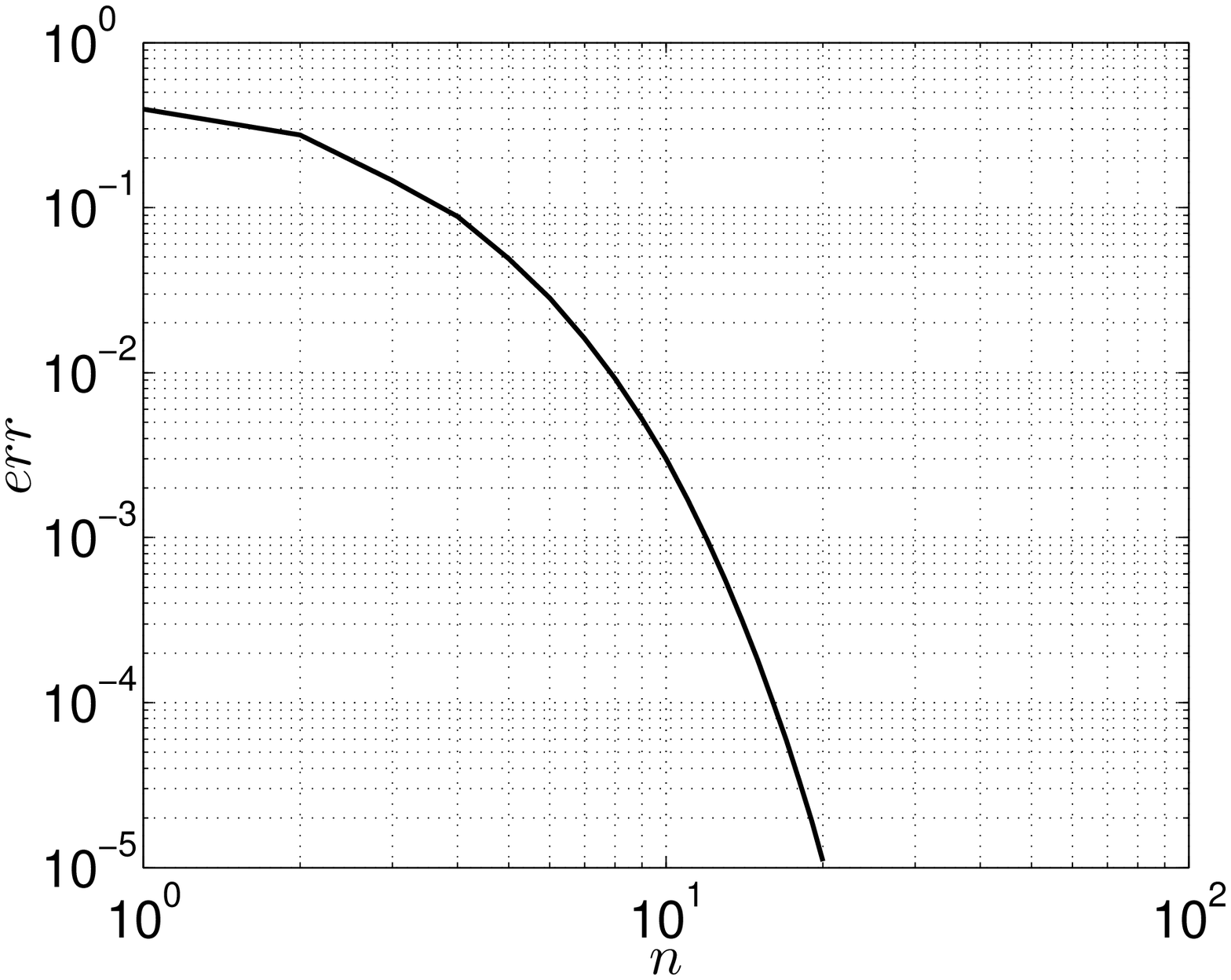}}
\caption{(a) Absolute error (on $L^2$-norm) of the numerical solution of eq. \eqref{exemplo2} in relation to number of integration points  $N$ putting the  iterations $n = 2, 10$, $20$, and using semi-log  scale; (b) comparison of numerical solution, and exact solution, in $10$ integration points; (c) absolute error (on $L^2$-norm) of the numerical solution  of eq.\eqref{exemplo2} on iterations number $n$ from $1$ to $20$ using log-log scale.}\label{funct_Ex2}
\end{figure}

\begin{table}[htbp] 
	\begin{center}
		\begin{tabular}{rrl}
			\hline
			Iteration (n) & Error in the $L^2$-norm (err)\\
			\hline 
1&  0.394366600397979\\
2&  0.276308698116164\\
3&  0.146122318109685\\
4&  0.881455213962e-1\\
5&  0.489253352732e-1\\
6&  0.283969210188e-1\\
10& 0.2994640228e-2\\
12& 0.97457513e-3\\
15& 0.18100617e-3\\
20& 0.109484e-4\\
			\hline 
		\end{tabular}
	\end{center}
	\caption{Error in the $L^2$-norm of the  approximate solution  with respect to eq. \eqref{exemplo2} for the iterations $n=1,2,\ldots,20$. }\label{tabEx2}
\end{table} 


\newpage

\section{Conclusion}\label{conclusion}

In this paper we expand the results of Emannuele \cite{Ema:92}  in the space  $L^p([a,b])$ for nonlinear integral equations through of  Theorem \ref{th_ex_uni}. We illustrate the guarantee of existence and uniqueness on the method of successive approximations considering  Chebyshev quadrature.  The computed errors for the exact solution and the iterated solution  present acceptable results.

\section*{Acknowledgements}

\noindent 
This work was supported by \emph{CNPq} (grants 441489/2014-1).


\section*{Appendix A\\Chebyshev polynomial method}
	In this section we apply  Chebyshev polynomial method (see \cite{PIESSENS1976})  for solving one dimensional functional-integral equation \eqref{emanu}. We  employ Chebyshev polynomial method of  the first kind to approximate functions  on the interval $[a, b]$. Firstly, we start with some basic definitions.
	\begin{definition}\label{Def_polichebychev}
	Chebyshev polynomials of degree $n$ are defined as:
		\begin{equation}
		{\displaystyle T_{n}(x)={\begin{cases}\cos {\big (}n\arccos x{\big )},&{\text{if }}|x|\leq 1,\\\cosh {\big (}n\operatorname {arcosh} x{\big )},&{\text{if }}x\geq 1,\\(-1)^{n}\cosh {\big (}n\operatorname {arcosh} (-x){\big )},&{\text{if }}x\leq -1.\end{cases}}} 
		\end{equation}
		In addition, these polynomials satisfy the following relations:
		$$\displaystyle T_{n}(\cos \theta )=\cos n\theta, \quad n = 0, 1, 2,\ldots $$
		and
		$${\displaystyle \int _{-1}^{1}\,{\frac {T_{n}(x)T_{m}(x)dx}{\sqrt {1-x^{2}}}}={\begin{cases}0,&n\neq m,\\\pi, &n=m=0,\\{\dfrac {\pi }{2},}&n=m\neq 0.\end{cases}}}$$
	\end{definition}
	
	\begin{remark} The set of Chebyshev polynomials form an orthogonal basis in $L^2([a,b])$, so that a function $f\in L^2([a,b])$  can be approximated via expansion as follows:
		\begin{equation}\label{approx_chebychev}
		\displaystyle f(x)\approx\sum_{n=0}^{M}a_{n}T_{n}\left(\frac{2}{b-a}x- \frac{b+a}{b-a} \right), \quad x \in [a,b],
		\end{equation}
		such that 
		\begin{equation}\label{coef_chebychev}
		a_n =  \displaystyle \frac{2}{\pi  d_k}\int _{-1}^{+1}{\frac {T_{n}(x)f(\frac{b-a}{2}x- \frac{b+a}{2})}{\sqrt {1-x^{2}}}}\,dx,\quad d_k ={\begin{cases}2,&n= 0,\\  1, &n>1.\end{cases}} 
		\end{equation}
	\end{remark}
	We estimate the unknown function $y(t)$  with the Chebyshev polynomials as
	\begin{equation}\label{approx_chebychev2}
	y(t)\approx y_M(t) = \sum_{n=0}^{M}c_{n}T_{n}(t).
	\end{equation}
	The unknown coefficients $c_{n}$  are determined by selecting  collocation points 
	$\{t_i\}_{i=0}^M$, where
	\begin{eqnarray}
	t_i = \frac{b-a}{2}x_i- \frac{b+a}{2},\quad x_i=  cos\left(\frac{i\pi}{M}\right).
	\end{eqnarray}
	The collocation method solves  the nonlinear integral equation \eqref{emanu} using approximation \eqref{approx_chebychev2}  through the equations:
	\begin{equation}\label{colocacaosistema}
	y_M(t_i) = f\left(t_i, \int_{a}^{b} k(t_i,s)g(s,y_M(s))ds\right) \qquad 0\leq i\leq M.
	\end{equation}
	%
	Now, by substituting the expression \eqref{coef_chebychev}   into \eqref{colocacaosistema}, we get the following  system: 
	\begin{equation}\label{colocacaosistema1}
	\sum_{n=0}^{M}c_{n}T_{n}(t_i) = f\left(t_i, \int_{a}^{b} k(t_i,s)g\left(s,\sum_{n=0}^{M}c_{n}T_{n}(s)\right)ds\right), \qquad 0\leq i\leq M,
	\end{equation}
	which in matrix form can be written in terms of the  vector $\B{c} =  [c_0, c_1,\ldots, c_M]^T$ as
	\begin{equation} \label{picard1}
	\B{T}\B{c} =  \B{F}(\B{c}),
	\end{equation}
	with  
	\begin{equation*}
	\B{F}(\B{c})=\left[F_0(\B{c}),\ldots,F_{n}(\B{c})\right]^T,
	\end{equation*}
	such that $$F_j(\B{c}) = \int_{a}^bk\left(t_j,s\right)g\left(s,\sum_{n=0}^{M}c_{n}T_{n}(s)\right)\,dx, \quad  0\leq j\leq M,$$
	and
	\begin{equation}\label{MatrixT}
	T_{i,j} = T_j(t_i),\quad  0\leq i,j\leq M. 
	\end{equation}
	This iterative process can be solved using Picard iteration method due to its nonlinearity. In this way, for each iteration we solve a linear problem: 
	\begin{equation} \label{iteracao1}
	\B{c}^{(k+1)}  =  \B{T}^{-1}\B{F}(\B{c}^{(k)}), \quad 0\leq k<k_{max},
	\end{equation}
	with $\B{c}^{(k)} = [c_0^{(k)}, c_1^{(k)},\ldots, c_M^{(k)}]^T$.
	The iterative process is stopped until the following stopping criterion is satisfied:
	\begin{equation} \label{iteracao3}
	 \|\B{c}^{(k+1)}-\B{c}^{(k)}\|_{p} < tol.
	\end{equation} 

\end{document}